\documentclass[12pt]{article}

\usepackage{setspace}

\usepackage{amsthm, amssymb, amstext}
\usepackage[fleqn]{amsmath}
\usepackage{latexsym}
\usepackage{comment}
\usepackage{hyperref}
\usepackage{mathtools}
\usepackage{enumerate} 

\usepackage{tikz, xcolor, graphicx, subcaption, float, bm}
\tikzstyle{uStyle}=[shape = circle, minimum size = 20pt, inner sep =2.5pt, outer sep = 0pt, draw, fill=white]

\usepackage{todonotes}
\usepackage{comment}

\newtheorem{theorem}{Theorem}
\newtheorem{lemma}{Lemma}

\newtheorem{conjecture}{Conjecture}

\newtheorem{claim}{Claim}
\newtheorem{subclaim}{Subclaim}

\theoremstyle{definition}

\usepackage{amsthm}

\title{Kempe Equivalent List Colorings Revisited}

\author{Dibyayan Chakraborty\thanks{School of Computing, University of Leeds, United Kingdom, email: \texttt{D.chakraborty@leeds.ac.uk} } \quad Carl Feghali\thanks{Univ. Lyon, EnsL, UCBL, CNRS, LIP, F-69342, Lyon Cedex 07, France, email: \texttt{feghali.carl@gmail.com} }  \quad Reem Mahmoud\thanks{Virginia Commonwealth University, Richmond, VA, USA, email: \texttt{mahmoudr@vcu.edu} }}

\date{}

\begin{document}
\maketitle

\begin{abstract}
A \emph{Kempe chain} on colors $a$ and $b$ is a component of the subgraph induced by colors $a$ and $b$. A \emph{Kempe change} is the operation of interchanging the colors of some Kempe chain. For a list-assignment $L$ and an $L$-coloring $\varphi$, a Kempe change is \emph{$L$-valid} for $\varphi$ if performing the Kempe change yields another $L$-coloring. Two $L$-colorings are \emph{$L$-equivalent} if we can form one from the other by a sequence of $L$-valid Kempe changes. A \emph{degree-assignment} is a list-assignment $L$ such that $L(v)\ge d(v)$ for every $v\in V(G)$. Cranston and Mahmoud (\emph{Combinatorica}, 2023) asked: For which graphs $G$ and degree-assignment $L$ of $G$ is it true that all the $L$-colorings of $G$ are $L$-equivalent?  We prove that for every 4-connected graph $G$ which is not complete and every degree-assignment $L$ of $G$, all $L$-colorings of $G$ are $L$-equivalent.  
 \end{abstract}


\newcommand{\List}[1]{L\left(#1\right)}

\noindent


\section{Introduction}

\emph{Reconfiguration} is the study of moving between different solutions of a problem using a predetermined transformation rule (a \emph{move}). Namely, given two solutions $A$ and $B$ of some underlying problem $X$, one asks whether there is a sequence of moves that transforms $A$ into $B$ so that, in addition, every intermediate term in the sequence is also a solution to $X$. 

Many reconfiguration problems arise from underlying decision problems on graphs. In this note, the underlying decision problem in question is the \emph{coloring problem}. Given a positive integer $k$ and a graph $G$, the coloring problem finds a (proper) \emph{$k$-coloring} of $G$. By $k$-coloring of a graph $G$, we mean a function $f: V(G) \rightarrow \{1, \dots, k\}$ such that $f(u) \not= f(v)$ if $uv \in E(G)$. A \emph{Kempe chain} on colors $a$ and $b$ is a component of the subgraph induced by colors $a$ and $b$. A \emph{Kempe change} is the operation of interchanging the colors of some Kempe chain. Two colorings are \emph{Kempe equivalent} if one can be formed from the other by a sequence of Kempe changes. The \emph{$k$-coloring reconfiguration problem} asks whether two given $k$-colorings of $G$ (i.e. two solutions of the coloring problem) are Kempe equivalent.  This notion of Kempe equivalence is inspired from a well-known attempt of Kempe at proving the Four Color Theorem and has been a popular topic of research since; see ~\cite{BBFJ, BHLN, cranston2022toroidal, las1981kempe, meyniel19785, mohar1, feghali} for some examples. It has also found applications in statistical physics~\cite{mohar2009new} and Markov chains~\cite{salas2022ergodicity}.  
Recently, Cranston and Mahmoud \cite{cranston2021kempe} investigated a coloring reconfiguration problem from a more general notion of coloring known as \emph{listing coloring}. A \emph{list-assignment} $L$ for a graph $G$ assigns a list of colors $L(v)$ for every $v\in V(G)$. A \emph{degree-assignment} is a list-assignment $L$ such that $L(v)\ge d(v)$ for every $v\in V(G)$. An \emph{$L$-coloring} of $G$ is a (proper) coloring $\varphi$ such that $\varphi(v)\in L(v)$ for every $v\in V(G)$. We say $G$ is \emph{$L$-colorable} if it admits an $L$-coloring. For a list-assignment $L$ and an $L$-coloring $\varphi$, a Kempe change is called \emph{$L$-valid} for $\varphi$ if performing the Kempe change yields another $L$-coloring. Two $L$-colorings are called \emph{$L$-equivalent} (or simply Kempe equivalent) if we can form one from the other by a sequence of $L$-valid Kempe changes. The \emph{list coloring reconfiguration problem} asks whether two given $L$-colorings of $G$ are $L$-equivalent.

Cranston and Mahmoud~\cite{cranston2021kempe} established the following result. 

\begin{theorem}[\cite{cranston2021kempe}]\label{thm:cranston}
Let $G$ be a connected graph and $L$ be a list-assignment for $G$. If $|L(x)|=\Delta(G)$ for every $x\in V(G)$ and $\Delta(G)\ge3$, then all $L$-colorings of $G$ are $L$-equivalent unless $L(v) = L(u)$ for every $u, v\in V(G)$ and $G$ is isomorphic to the complete graph or the triangular prism. 
\end{theorem}

Towards a strengthening of Theorem \ref{thm:cranston}, Cranston and Mahmoud also asked to characterize all degree-swappable graphs~\cite{cranston2021kempe}, where  
a graph $G$ is \emph{degree-swappable} if for any degree-assignment $L$, all $L$-colorings of $G$ are $L$-equivalent. The purpose of this note is to address this problem.

\begin{theorem}\label{thm:main2}
Every $4$-connected graph distinct from the complete graph is degree-swappable. 
\end{theorem}

In particular, our Theorem \ref{thm:main2} strengthens Theorem \ref{thm:cranston} for $4$-connected graphs. Note that Cranston and Mahmoud~\cite{cranston2021kempe} gave examples of $2$-connected graphs which are not degree-swappable. We conjecture the following.

\begin{conjecture}
Every $3$-connected graph distinct from the complete graph and the triangular prism is degree-swappable.    
\end{conjecture}

The paper is organized as follows. In Section \ref{sec:lem}, we state some lemmas that will be used repeatedly in the proof of Theorem \ref{thm:main2}. In Section \ref{sec:idea}, we give the proof idea. Finally, in Section \ref{sec:main}, we give the proof of Theorem \ref{thm:main2}.  

\section{Key lemmas}\label{sec:lem}

A \emph{Gallai tree} is a connected graph $T$ such that every 2-connected block
of $T$ is either a clique or an odd cycle. Suppose $B_1, \dots, B_k$ are the blocks of
a Gallai tree $T$, and let $S_1, \dots, S_k$ be sets of colors satisfying the following
conditions:
\begin{itemize}
\item For $1 \leq i \leq k$,  if $B_i$ is a clique, then $|S_i| = |V(B_i)| - 1$, and if $B_i$ is an odd cycle, then $|S_i| = 2$.
\item For $1 \leq i < j \leq k$, if $B_i \cap B_j \not= \emptyset$, then $S_i \cap S_j = \emptyset$. 
\end{itemize}
If $L(v) = \bigcup_{v \in V(B_i)} S_i$ for $v \in V(T)$, then $L$ is called \emph{blockwise uniform}. 

We shall repeatedly use the following two lemmas, the first lemma being a classical result in graph coloring. 

\begin{lemma}[\cite{gallai1963kritische, borodin1979problems,erdos1979choosability}]\label{lem:borodin}
Let $G$ be a connected graph and $L$ be a degree-assignment for $G$. The graph $G$ is not $L$-colorable if and only if $G$ is a Gallai tree and $L$ is blockwise uniform. 
\end{lemma}

For a vertex $v \in V(G)$ and a list-assignment $L$,  call $v$  \emph{special} (with respect to $L$) if $v$ has a neighbor $u$ such that $L(v) \setminus L(u) \not= \emptyset$. In this case, $u$ is called a \emph{special neighbor of $v$} and every color in $L(v) \setminus L(u)$ is called \emph{special for $v$}. 

\begin{lemma}\label{lem:equiv2}
    Let $G$ be a $3$-connected graph and $L$ be a degree-assignment for $G$. If $x$ and $y$ are non-adjacent special vertices in $G$ with special colors $a$ and $b$, respectively (possibly $a=b$), then all $L$-colorings of $G$ that either assign color $a$ to $x$ or color $b$ to $y$ are $L$-equivalent.
\end{lemma}

To prove Lemma \ref{lem:equiv2}, we require the following two auxiliary results. 

\begin{lemma}[\cite{cranston2021kempe,las1981kempe}]\label{lem:cranston}
Let $G$ be a connected graph and $L$ be a degree-assignment for $G$. If $|L(x)| > d(x)$ for some $x \in V(G)$, then all $L$-colorings of $G$ are $L$-equivalent. 
\end{lemma}


\begin{lemma}[\cite{cranston2021kempe}]\label{lem:equivalent1}
Let $G$ be a $2$-connected graph and $L$ be a degree-assignment for $G$. If $x$ is a special vertex in $G$, then for every special color $a$ of $x$, all $L$-colorings of $G$ that assign color $a$ to $x$ are $L$-equivalent. 
\end{lemma}

\begin{proof}
    Let $L'$ be the list-assignment for $G - x$ defined by $L'(w) = L(w) \setminus \{a\}$ if $w \in N(x)$ and $L'(w) = L(w)$ otherwise. Let $z$ be a special neighbor of $x$; that is, $a\notin L(z)$. Observe that $L'$ is a degree-assignment for $G-x$ and that $|L'(z)| > d_{G - x}(z)$. Since $G - x$ is connected (as $G$ is $2$-connected), all $L'$-colorings of $G - x$ are $L'$-equivalent, by Lemma \ref{lem:cranston}. Thus, all $L$-colorings of $G$ that assign color $a$ to $x$ are $L$-equivalent. 
\end{proof}


\begin{proof}[Proof of Lemma \ref{lem:equiv2}]
     Let $L_1$ and $L_2$ be the set of $L$-colorings of $G$ that assign color $a$ to $x$ and $b$ to $y$, respectively. By Lemma \ref{lem:equivalent1}, all $L$-colorings of $L_1$ are $L$-equivalent. Similarly, all $L$-colorings of $L_2$ are $L$-equivalent. Thus, it remains to show that every $L$-coloring of $L_1$ is $L$-equivalent to every $L$-coloring of $L_2$. This is achieved by showing that $G$ admits an $L$-coloring $\varphi$ with $\varphi\in L_1\cap L_2$. Indeed, every $L$-coloring of $L_1$ is therefore $L$-equivalent to every $L$-coloring of $L_2$ via $\varphi$. 
     
     Let $z$ be a special neighbor of $x$; that is, $a\notin L(z)$. Observe that $z\neq y$ since $x$ and $y$ are non-adjacent. Let $L'$ be the list-assignment for $G - \{x,y\}$ defined by
     \begin{itemize}
         \item  $L'(w) = L(w) \setminus \{a\}$ if $w \in N(x) \setminus N(y)$,
         \item  $L'(w) = L(w) \setminus \{b\}$ if $w \in N(y) \setminus N(x)$,
         \item  $L'(w) = L(w) \setminus \{a,b\}$ if $w \in N(x) \cap N(y)$, and
         \item $L'(w) = L(w)$ otherwise. 
     \end{itemize}
    Observe that $L'$ is a degree-assignment for $G-\{x,y\}$ and $|L'(z)|>d_{G - \{x,y\}}(z)$. Further, $G-\{x,y\}$ is connected, since $G$ is 3-connected. Let $T$ be a spanning tree of $G - \{x, y\}$ rooted at $z$. List the vertices of $G - \{x,y\}$ in 
     order of decreasing distance from $z$, then color them greedily from $L'$ following this order to obtain an $L'$-coloring $\varphi'$ of $G-\{x,y\}$. This is always possible since each vertex $w\neq z$ in the ordering has at least one uncolored neighbor (its parent in $T$), so $w$ has a color in $L'(w)$ which does not appear on its already colored neighbors; moreover, $z$ has a color in $L'(z)$ which does not appear on its neighbors since $|L'(z)|>d_{G - \{x,y\}}(z)$. Finally, we extend $\varphi'$ to $\varphi \in L_1 \cap L_2$ by coloring $x$ and $y$ with $a$ and $b$, respectively. 
\end{proof}

\section{Outline of the proof}\label{sec:idea}

The main idea in the proof of Theorem \ref{thm:main2} is simple. It consists in making it possible to apply Lemma \ref{lem:equivalent1}. To be able to do so, note, thanks to Theorem \ref{thm:cranston}, that one can assume $|L(v)| \not= \Delta(G)$ for some $v \in V(G)$. This along with the fact that $L$ is a degree-assignment for $G$ allows us to assume, by the pigeonhole principle, that $v$ is special. It is then perhaps evident how Lemma \ref{lem:equivalent1} becomes applicable: If we can show that every $L$-coloring of $G$ is $L$-equivalent to some $L$-coloring that assigns the special color of $v$, say $a$, to $v$, then Theorem \ref{thm:main2} follows. There is one main challenge in doing this: For an $L$-coloring $\varphi$ of $G$, if $\varphi(v) = a$ or if $\varphi(v) \not=a$ but applying a Kempe change at the Kempe chain on colors $a$ and $\varphi(v)$ that contains $v$ is $L$-valid, then our aim is obviously achieved. The proof thus reduces to showing that, for every $L$-coloring $\varphi$ such that $\varphi(v) \not=a$ and for which the Kempe change that interchanges color $a$ and $\varphi(v)$ at $v$ is not $L$-valid, $\varphi$ is, nevertheless, Kempe equivalent, via a sequence of (possibly several) Kempe changes, to some $L$-coloring that assigns color $a$ to $v$. 

Precisely, let $H_{a, \varphi(v)}(v)$ denote the Kempe chain on colors $a$ and $\varphi(v)$ that contains $v$. Since a Kempe change on $H_{a, \varphi(v)}(v)$ is not L-valid, there exists a vertex $u$ different from $v$ in $H_{a, \varphi(v)}(v)$ such that $$(\{a, \varphi(v)\} \setminus \{\varphi(u)\}) \notin L(u).$$ We consider three cases depending on the distance of $u$ to $v$ in $H_{a, \varphi(v)}(v)$:
\begin{itemize}
    \item[(1)] If $u$ is a neighbor of $v$ (see Claim \ref{claim:1})
    \item[(2)] If $u$ is at distance $2$ from $v$ (see Claims \ref{claim:2} and \ref{claim:3}), and
    \item[(3)] If $u$ is at distance at least $3$ from $v$ (see the end of the proof). 
\end{itemize}
To be slightly more precise, Claim \ref{claim:1} essentially reduces case (1) to (2) or (3) while Claims \ref{claim:2} and \ref{claim:3} treat case (2). To finish off the proof, Lemma \ref{lem:equiv2} is then evoked towards the end to settle case (3).

\section{The proof of Theorem \ref{thm:main2}}\label{sec:main}






Let $G$ be a graph with list-assignment $L$. For a vertex $v\in V(G)$, denote by \emph{$N^k(v)$} the set of vertices at distance (exactly) $k$ from $v$ (with $N^1(v)=N(v)$). 

Fix an $L$-coloring $\varphi$ of $G$, a vertex $v\in V(G)$, and a color $a\in L(v)$ with $a\neq \varphi(v)$. Let \emph{$H_{\varphi(v), a}(v)$} denote the Kempe chain on colors $\varphi(v)$ and $a$ under $\varphi$ that contains $v$. (By definition, the vertices in $N^i(v)\cap H_{\varphi(v), a}(v)$ are colored with $a$ and with $\varphi(v)$ for all odd and even values, respectively, of $i$.)  Let
\begin{itemize}
    \item $N^1_v(\varphi(v),a) = \{u: u \in N(v) \cap H_{\varphi(v), a}(v),  \varphi(v) \not\in L(u)\}$
    \item $ N^2_v(\varphi(v), a) = \{u: u \in N^2(v) \cap H_{\varphi(v), a}(v),  a \not\in L(u)\}$
    \item $N^{3^+}_v(\varphi(v), a) = \{u: u \in N^k(v) \cap H_{\varphi(v), a}(v),  \{\varphi(v),a\}\nsubseteq L(u), k\ge3\}$
\end{itemize}

\begin{proof}[Proof of Theorem \ref{thm:main2}]
 Since $G$ is 4-connected, $\Delta(G)\ge\delta(G)\ge4$ and so $G$ is not isomorphic to the triangular prism. If $L$ is identical everywhere (that is, $|L(v)|\ge\Delta(G)$ for every $v\in V(G)$), then we are done by Lemma~\ref{lem:cranston} or Theorem~\ref{thm:cranston}. Thus, we can assume that $L$ is not identical everywhere. So, there exist adjacent $u,v\in V(G)$ and a color $a\in L(v)$ such that $a \not\in L(u)$; hence, $v$ is special, $a$ is special for $v$, and $u$ is a special neighbor of $v$. Obviously, we can further assume $v$ has maximum degree in $G$. Let $L_1$ be the set of all $L$-colorings of $G$ that assign color $a$ to $v$. By Lemma \ref{lem:equivalent1}, all $L$-colorings of $L_1$ are Kempe equivalent. 
 
 Our aim in the remainder of the proof is to show that every $L$-coloring $\varphi\notin L_1$ is $L$-equivalent to some $L$-coloring in $L_1$. Now, if we can perform a Kempe change on $H_{\varphi(v), a}(v)$ starting from $\varphi$, our aim is obviously achieved. So, we can assume this Kempe change is not $L$-valid.  This implies either \begin{itemize}
    \item $N^1_v(\varphi(v),a) \not= \emptyset$,
    \item $N^2_v(\varphi(v),a) \not=\emptyset$, or
    \item $N^{3^+}_v(\varphi(v),a)\not=\emptyset$. 
\end{itemize}

We treat each of the above three cases separately in the following three claims, from which Theorem \ref{thm:main2} follows rather easily. 

\begin{claim}\label{claim:1}
    Fix any $x\in V(G)$ and any $L$-coloring $\psi$ of $G$. If
    \begin{itemize}
        \item[(i)] $d(x)=\Delta(G)$ or
        \item[(ii)] there exists a vertex $y$ non-adjacent to $x$ with special color $c\neq\psi(x)$,
    \end{itemize}
 then $\psi$ is $L$-equivalent to some $L$-coloring $\phi$ of $G$ such that $\phi(x) = \psi(x)$ and $N^1_x(\phi(x), c) = \emptyset$. 
\end{claim}

\begin{proof}
    If $N^1_x(\psi(x),c)=\emptyset$, then the statement is trivially true (with $\phi=\psi$). If $N^1_x(\psi(x),c)\neq\emptyset$, let $L'$ be the list-assignment for $G - x$ defined by 

$$
L'(w)=\left\{
	\begin{array}{ll}
		L(w) \setminus \{\psi(x)\} & \textrm{if $w \in N(x)$ and $\psi(x) \in L(w)$,}\\
		L(w) \setminus \{c\} \quad & \textrm{if $w \in N(x)$ and $\psi(x) \not\in L(w)$},\\
		L(w) & \textrm{otherwise}
	\end{array}\right.
$$
for each $w \in V(G-x)$.

\begin{subclaim}
The graph $G - x$ has an $L'$-coloring. 
\end{subclaim}
\begin{proof}
   Note that $L'$ is a degree-assignment for $G - x$. Moreover, $G-x$ is not an odd cycle since it is 3-connected. Now $G - x$ has an $L'$-coloring unless it is a complete graph and $L'$ is identical everywhere, by Lemma~\ref{lem:borodin}.  If (i) is true, then $G - x$ is not a complete graph; otherwise, $G$ is also a complete graph. If instead (ii) is true, then observe that $c\in L'(y)$ since $y$ is non-adjacent to $x$. Further, $c\notin L'(z)$ for some special neighbor $z$ of $y$. So, $L'$ is not identical everywhere. Thus, there exists an $L'$-coloring of $G-x$.  
\end{proof}  
Let $\phi'$ be an $L'$-coloring of $G- x$, and let $\phi$ be the $L$-coloring of $G$ such that $\phi(x) = \psi(x)$ and $\phi \restriction (G - x) = \phi'$ (where $\phi \restriction (G - x)$ denotes the restriction of $\phi$ to $G - x$). Now $\phi$ is proper and $N^1_x(\phi(x), c)= \emptyset$. Since $N^1_x(\psi(x), c) \not= \emptyset$, the color $\psi(x)$ is special for $x$. By Lemma~\ref{lem:equivalent1}, all $L$-colorings that assign color $\varphi(x)$ to $x$ are $L$-equivalent, i.e., $\psi$ is $L$-equivalent to $\phi$.   
\end{proof}



\begin{claim}\label{claim:2}
    If $N^1_v(\varphi(v), a) = \emptyset$ but $N^2_v(\varphi(v), a) \not= \emptyset$, then at least one of the following holds:
    \begin{itemize}
        \item $\varphi$ is $L$-equivalent to an $L$-coloring $\phi$ of $G$ such that $\phi(v) = \varphi(v)$ and $N^1_v(\phi(v), a) = N^2_v(\phi(v), a) = \emptyset$.
        \item There exist non-adjacent special vertices $x,y\in V(G)$ such that $a$ is special for both $x$ and $y$. 
    \end{itemize}
\end{claim}

\begin{proof}
Since  $N^1_v(\varphi(v), a) = \emptyset$ but $N^2_v(\varphi(v), a) \not= \emptyset$, there exist $x,z\in H_{\varphi(v), a}(v)$ such that, by definition,
\begin{itemize}
    \item $x\in N(v)$ and $z\in N(x)\cap N^2(v)$, and
    \item $\varphi(x) = a$ and $a \not\in L(z)$.
\end{itemize}
See also Figure~\ref{fig2}. In particular, $a$ is special for $x$. Thus, by Lemma \ref{lem:equivalent1}, all $L$-colorings that assign color $a$ to $x$ are $L$-equivalent. Let $L'$ be the list-assignment for $G - \{v,x\}$ defined by
$$
L'(w)=\left\{
	\begin{array}{ll}
		L(w) \setminus \{\varphi(v)\} & \textrm{if $w \in N(v) \setminus N(x)$ and $\varphi(v) \in L(w)$,}\\
		L(w) \setminus \{a\} \quad & \textrm{if $w \in N(v) \setminus N(x)$ and $\varphi(v) \not\in L(w)$},\\
  L(w) \setminus \{a\} \quad & \textrm{if $w \in N(x) \setminus N(v)$ and $a \in L(w)$},\\
  L(w) \setminus \{\varphi(v)\} \quad & \textrm{if $w \in N(x) \setminus N(v)$ and $a \not\in L(w)$},\\
  L(w) \setminus \{a, \varphi(v)\} \quad & \textrm{if $w \in N(v) \cap N(x)$},\\
		L(w) & \textrm{otherwise}
	\end{array}\right.
$$
for every $w \in V(G -\{v,x\})$.

\begin{subclaim}
$G - \{v,x\}$ has an $L'$-coloring $\phi'$.
\end{subclaim}

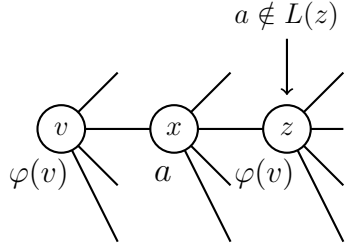
\begin{figure}[t]
\centering
\begin{tikzpicture}[rotate=90,scale=1.5, every node/.style={scale=0.9}]
\draw[thick] (0,0) -- (0,-1) (-0.5,-0.5) -- (0,0) -- (-1,-0.5) (0,0) -- (0.5,-0.5) (0,-1) -- (0,-2) (-0.5,-1.5) -- (0,-1) -- (-1,-1.5) (0,-1) -- (0.5,-1.5) (-0.5,-2.5) -- (0,-2) -- (-1,-2.5) (0,-2.5) -- (0,-2) -- (0.5,-2.5);
\draw (0,-3) node[uStyle, scale=0.7, draw=white] {};
\draw[thick] (0,0) node[uStyle] {$v$} (0,-1) node[uStyle] {$x$} (0,-2) node[uStyle] {$z$} (-0.4,0.2) node[scale=1.1] {$\varphi(v)$} (-0.4,-0.9) node[scale=1.2] {$a$} (-0.4,-1.8) node[scale=1.1] {$\varphi(v)$} (1,-2) node {$a\notin L(z)$}; 
\draw[thick,<-] (0.3,-2) -- (0.8,-2);
\end{tikzpicture}
\caption{$N^2(\varphi(v),a)\neq\emptyset$, so there exists $z$ with $a\notin L(z)$.}
\label{fig2}
\end{figure} 

\begin{proof}
 Note that $L'$ is a degree-assignment for $G - \{v,x\}$. Moreover, recall that $G$ is 4-connected, so $d_G(w)\ge4$ for every $w\in V(G)$. This implies $G - \{v,x\}$ is not a cycle; otherwise, $z\in N(v)\cap N(x)$, contradicting that $\varphi(v)=\varphi(z)$. Therefore, $G - \{v,x\}$  has an $L'$-coloring unless it is complete, by Lemma \ref{lem:borodin}. But if $G - \{v,x\}$ is complete, then $v$ is adjacent to every vertex of $G - v$, since $d(v)=\Delta(G)$. This contradicts the fact that $v$ and $z$ are colored alike. Thus, $G-\{v,x\}$ is not complete, as required.
\end{proof}
Let $\phi$ be the $L$-coloring of $G$ defined by $\phi(v)=\varphi(v)$, $\phi(x)=\varphi(x)=a$, and $\phi \restriction (G-\{v,x\})=\phi'$. Observe that $\varphi$ is $L$-equivalent to $\phi$ since $\phi(x)=\varphi(x)=a$. Further, by the definition of $L'$, it follows that $N^1_v(\phi(v), a) = \emptyset$ and no neighbor of $x$ belongs to $N^2_v(\phi(v), a)$. So, if $N^2_v(\phi(v), a) \not= \emptyset$, then there exist $y,z'\in H_{\phi(v),a}(v)$ with $y\neq x$ and $z'\neq z$ such that 
\begin{itemize}
    \item $y \in N(v)$ and $\phi(y) = a$,
    \item $z'\in N(y)\cap N^2(v)$ and $\phi(z') = \phi(v)$, and 
    \item $a\notin L(z')$.
\end{itemize} 
Thus, $x$ and $y$ are non-adjacent special vertices with $a$ as a common special color. This completes the proof of the claim.
\end{proof}

\begin{claim}\label{claim:3}
If there exist non-adjacent special vertices $x,y\in V(G)$ with some common special color $c$, then all $L$-colorings of $G$ are $L$-equivalent. 
\end{claim}

\begin{proof}  
Let $L_1$ be the set of $L$-colorings of $G$ that assign $c$ to $x$ or to $y$. By Lemma \ref{lem:equiv2}, all $L$-colorings of $L_1$ are $L$-equivalent. So, it suffices to show that every $\phi\notin L_1$ is $L$-equivalent to some $L$-coloring in $L_1$. Note that this is trivially true if performing a Kempe chain on $H_{\phi(x),c}(x)$ is $L$-valid. Otherwise, at least one of $N^1_x(\phi(x),c)$, $N^2_x(\phi(x),c)$, or $N^{3+}_x(\phi(x),c)$ is nonempty. By Claim~\ref{claim:1}(ii), we can assume that $N^1_x(\phi(x), c) = \emptyset$ since $y$ is non-adjacent to $x$ and $c$ is special for $y$. We split the rest of the proof into two cases. 

\begin{itemize}
    \item[Case 1:]  $N^2_x(\phi(x), c) = \emptyset.$

    In this case, $N_x^{3^+}(\phi(x),c)\neq\emptyset$. Fix $t\in N_x^{3^+}(\phi(x),c)$ and fix $\beta\in\{\phi(x),c\}$ with $\beta\notin L(t)$. Note that the neighbor $z$ of $t$ on a shortest path from $x$ to $t$ in $H_{\phi(x),c}(x)$ is non-adjacent to $x$. Moreover, $\phi(z)=\beta$ and $\beta$ is special for $z$; see Figure~\ref{fig1}. Let us first show that there exists an $L$-coloring $\psi\in L_1$ which assigns $c$ to $x$. Indeed, if we let $L'$ be the list-assignment for $G-x$ defined by $L'(w)=L(w)\setminus\{c\}$ for every $w\in N(x)$ and $L'(w)=L(w)$ for every other vertex, then $L'$ is a degree-assignment for $G-x$. Further, since $G-x$ is connected and $c$ is special for $x$, we have $|L'(r)|> d_{G-x}(r)$ for some special neighbor $r$ of $x$. As in the proof of Lemma~\ref{lem:equiv2}, we color the vertices of $G-x$ in order of decreasing distance from $r$ to obtain an $L'$-coloring $\psi'$ of $G-x$. Finally, we extend $\psi'$ to $\psi$ by coloring $x$ with $c$. Now since $x$ and $z$ are non-adjacent with special colors $c$ and $\beta$, respectively, $\phi$ is $L$-equivalent to $\psi$, by Lemma~\ref{lem:equiv2}. Thus, all $L$-colorings of $G$ are $L$-equivalent.

\item[Case 2:] $N^2_x(\phi(x), c) \not= \emptyset$

In this case, there exists a neighbor $z$ of $x$ with $\phi(z)=c$ and a neighbor $t$ of $z$ with $\phi(t) = \phi(x)$ and $c \not\in L(t)$; see Figure~\ref{fig1}. So, $t \not= y$. Let $H:= G - \{x, t\}$ and let $L_H$ be the list-assignment for $H$
 defined by 
$$
L_H(w)=\left\{
	\begin{array}{ll}
		L(w) \setminus \{\phi(x)\} & \textrm{if $w \in N(x) \cup N(t)$,}\\
		L(w) & \textrm{otherwise}
	\end{array}\right.
$$
for each $w \in V(H)$. To see that  $H$ admits an $L_H$-coloring which colors $y$ with $c$, observe that $L_H$ is a degree-assignment for $H$ and $|L_H(z)| > d_H(z)$. Moreover, since $G$ is $4$-connected, $H-y$ is connected. So, there is a spanning tree $T$ of $H - y$ that is rooted at $z$. As in the proof of Lemma~\ref{lem:equiv2}, we order the vertices of $H$ by placing $y$ first and then placing the other vertices of $T$ in order of decreasing distance from $z$. Now $H$ has an $L_H$-coloring $\psi_H$ that assigns $c$ to $y$, as desired. 

By Lemma \ref{lem:cranston}, all $L_H$-colorings of $H$ are $L_H$-equivalent. So, $\psi_H$ is $L_H$-equivalent to $\phi\restriction H$. Let $\psi$ be the $L$-coloring of $G$ defined by $\psi(x)=\psi(t)=\phi(x)$ and $\psi\restriction H=\psi_H$. Now $\phi$ is $L$-equivalent to $\psi$. Thus, all $L$-colorings of $G$ are $L$-equivalent. This completes Case 2. 
\end{itemize}
The claim is proved. 
\end{proof}

\begin{figure}[t]
\begin{subfigure}{0.55\textwidth}
\centering
\begin{tikzpicture}[scale=1.5, every node/.style={scale=0.9}]
\draw[thick] (0,0) -- (0,-1) (-0.5,-0.5) -- (0,0) -- (-1,-0.5) (0,0) -- (0.5,-0.5) (0,-1) -- (0,-2) (0,-3) -- (0,-2) (-0.5,-2.5) -- (0,-2) -- (-1,-2.5) (0,-2) -- (0.5,-2.5);
\draw[thick] (0,0) node[uStyle] {$x$} (0,-0.5) node[uStyle, scale=0.5] {} (0,-1) node[uStyle, scale=0.5] {} (0,-1.5) node[uStyle, scale=0.5] {} (0,-2) node[uStyle] {$z$} (0,-3) node[uStyle] {$t$} (-0.6,0.1) node[scale=1.1] {$\phi(x)$} (-0.4,-1.9) node[scale=1.1] {$\beta$} (1.4,-3) node {$\beta\notin L(t)$};
\draw[thick,<-] (0.3,-3) -- (0.8,-3);
\end{tikzpicture}
\end{subfigure}%
\begin{subfigure}{0.45\textwidth}
\centering
\begin{tikzpicture}[scale=1.5, every node/.style={scale=0.9}]
\draw[thick, white] (0,-3) node[uStyle] {}; 
\draw[thick] (0,0) -- (0,-1) (-0.5,-0.5) -- (0,0) -- (-1,-0.5) (0,0) -- (0.5,-0.5) (0,-1) -- (0,-2) (-0.5,-1.5) -- (0,-1) -- (-1,-1.5) (0,-1) -- (0.5,-1.5) (-0.5,-2.5) -- (0,-2) -- (-1,-2.5) (0,-2.5) -- (0,-2) -- (0.5,-2.5);
\draw[thick] (0,0) node[uStyle] {$x$} (0,-1) node[uStyle] {$z$} (0,-2) node[uStyle] {$t$} (-0.6,0.1) node[scale=1.2] {$\phi(x)$} (-0.4,-0.9) node[scale=1.2] {$c$} (-0.6,-1.9) node[scale=1.2] {$\phi(x)$} (1.35,-2) node {$c\notin L(t)$}; 
\draw[thick,<-] (0.3,-2) -- (0.8,-2);
\end{tikzpicture}
\end{subfigure}
\caption{(Left) Case 1: $t\in N^{3^+}(\phi(x),c)$ and $\beta\notin L(t)$. (Right) Case 2: $t\in N^2(\phi(x),c)$, so $c$ is special for $z$.}
\label{fig1}
\end{figure}
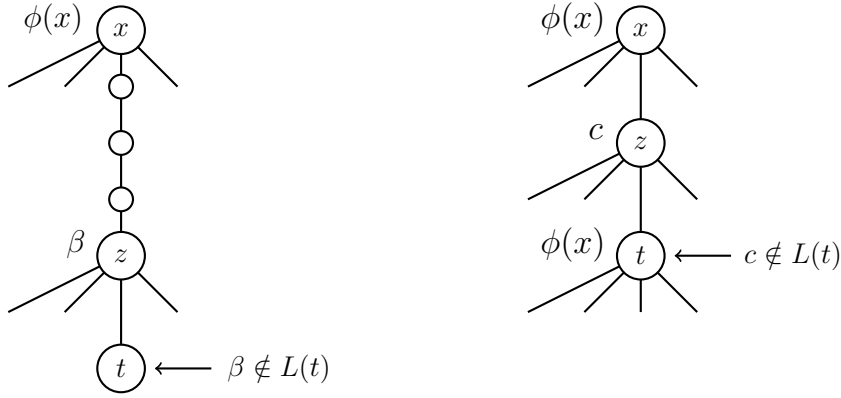

To finish the proof: By Claim~\ref{claim:1}(i), we have $N^1_v(\varphi(v),a)=\emptyset$. By Claim~\ref{claim:2}, either $N^2_v(\varphi(v),a)=N^1_v(\varphi(v),a)=\emptyset$ or there exist non-adjacent vertices $x$ and $y$ with a common special color. If the latter is true, we are done by Claim~\ref{claim:3}. If the former is true, recall that at least one of $N^1_v(\varphi(v),a)$, $N^2_v(\varphi(v),a)$, or $N^{3^+}_v(\varphi(v),a)$ is nonempty. Thus, $N^{3^+}_v(\varphi(v),a)\neq\emptyset$. This implies $H_{\varphi(v), a}(v)$ contains a vertex $z$ non-adjacent to $v$ such that $\varphi(z)$ is special for $z$. Recall that $a$ is special for $v$. As in the proof of Lemma~\ref{lem:equiv2} (or similarly, Case 1 in Claim~\ref{claim:3}), it is easy to show that there exists an $L$-coloring $\phi$ of $G$ which assigns the color $a$ to $v$. Now $\varphi$ is $L$-equivalent to $\phi$, by Lemma~\ref{lem:equiv2}. Thus, all $L$-colorings of $G$ are $L$-equivalent. 
\end{proof}

\section*{Acknowledgements}
We thank both referees for many helpful suggestions, which led to an improvement in the presentation, and for spotting some inaccuracies. We also thank Cl\'ement  Legrand-Duchesne and St\'ephan Thomass\'e for helpful discussions. 
The second author was supported by Agence
Nationale de la Recherche (France) under research grant ANR DIGRAPHS ANR-19-CE48-
0013-01.
1

\bibliographystyle{abbrv}

\end{document}